\documentclass[leqno,bibother]{amsart}

\usepackage{amssymb, graphicx}
\usepackage{dsfont}
\usepackage{bbm}
\usepackage[normalem]{ulem}
\usepackage{braket}
\usepackage{enumitem}

\usepackage{tikz}        
\usetikzlibrary{matrix, arrows}

\usepackage{xcolor}

\usepackage{aliascnt}

\newtheorem{theorem}{Theorem}[section]

\newaliascnt{lemma}{theorem}
\newtheorem{lemma}[lemma]{Lemma}
\aliascntresetthe{lemma}

\newaliascnt{remark}{theorem}
\newtheorem{remark}[remark]{Remark}
\aliascntresetthe{remark}

\newaliascnt{corollary}{theorem}
\newtheorem{corollary}[corollary]{Corollary}
\aliascntresetthe{corollary}

\newaliascnt{proposition}{theorem}
\newtheorem{proposition}[proposition]{Proposition}
\aliascntresetthe{proposition}

\newaliascnt{question}{theorem}
\newtheorem{question}[question]{Question}
\aliascntresetthe{question}

\newaliascnt{fact}{theorem}
\newtheorem{fact}[fact]{Fact}
\aliascntresetthe{fact}

\newaliascnt{conjecture}{theorem}
\newtheorem{conjecture}[conjecture]{Conjecture}
\aliascntresetthe{conjecture}

\theoremstyle{definition}

\newaliascnt{example}{theorem}
\newtheorem{example}[example]{Example}
\aliascntresetthe{example}

\newaliascnt{definition}{theorem}
\newtheorem{definition}[definition]{Definition}
\aliascntresetthe{definition}

\usepackage{hyperref}
\hypersetup{
	colorlinks,
	linkcolor={red!60!black},
	citecolor={blue!60!black},
	urlcolor={blue!80!black}
}

\usepackage{cleveref}

\crefname{question}{question}{questions}
\Crefname{question}{Question}{Questions}

\crefname{conjecture}{conjecture}{Conjectures}
\crefname{fact}{fact}{facts}

\newcommand{\cupdot}{\mathbin{\mathaccent\cdot\cup}}

\newcommand{\twopartdef}[4]
{
	\left\{
	\begin{array}{ll}
		#1 & \mbox{if } #2 \\
		#3 & \mbox{if } #4
	\end{array}
	\right.
}

\newcommand{\cM}{\mathcal{M}}
\newcommand{\cN}{\mathcal{N}}
\newcommand{\cL}{\mathcal{L}}

\newcommand{\cU}{\mathcal{U}}
\newcommand{\cK}{\mathcal{K}}

\newcommand{\bN}{\mathbb{N}}
\newcommand{\bZ}{\mathbb{Z}}
\newcommand{\bR}{\mathbb{R}}

\newcommand{\vphi}{\varphi}

\DeclareMathOperator{\rank}{rank}
\DeclareMathOperator{\Def}{Def}

\newcommand{\iiff}{\leftrightarrow}

\title{Pseudo-finite sets, pseudo-o-minimality}
\author[Nadav Meir]{Nadav Meir} \thanks{The work in this paper is part of the author's Ph.D. studies at the Department of Mathematics, Ben-Gurion University of the Negev under the supervision of Assaf Hasson.} \thanks{The author was partially supported by ISF Grant 181/60 and the Hillel Gauchman scholarship.}
\email{mein@math.bgu.ac.il}
\address{Department of Mathematics,
	Ben Gurion University of the Negev \\
	P.O.B. 653, 
	Be'er Sheva 8410501, Israel.
}

\makeatletter
\newcommand{\leqnomode}{\tagsleft@true}
\newcommand{\reqnomode}{\tagsleft@false}
\makeatother


\begin{document}
	\reqnomode
	\begin{abstract}	
We give an example of two ordered structures $\cM,\cN$ in the same language $\cL$ with the same universe, the same order and admitting the same one-variable definable subsets such that $\cM$ is a model of the common theory of o-minimal $\cL$-structures and $\cN$ admits a definable, closed, bounded, and discrete subset and a definable injective self-mapping of that subset which is not surjective. This answers negatively two questions by Schoutens; the first being whether there is an axiomatization of the common theory of o-minimal structures in a given language by conditions on one-variable definable sets alone. The second being whether definable completeness and type completeness imply the pigeonhole principle. It also partially answers a question by Fornasiero asking whether definable completeness of an expansion of a real closed field implies the pigeonhole principle.  
	\end{abstract}
	\maketitle

\setcounter{theorem}{0}

\section{Introduction}
o-minimality is not preserved under ultraproducts, as shown in the following example:
\begin{example}\label{exampleNonUltra}
	Let $\cL=\Set{<, U}$ where $<$ is a binary relation symbol and $U$ is a unary predicate.
	For every $n\in \bN$, let $\cM_n$ be a structure interpreting $<$ as a dense linear order without end points and $U$ as a set of points of size $n$. Then each $\cM_n$ is o-minimal. But for any non-principal ultrafilter $\cU$ on $\bN$, in the ultraproduct $\prod_{\bN} \cM_n / \cU$, the definable set $U$ is infinite and discrete, thus the ultraproduct of o-minimal structures need not be o-minimal.
\end{example}
\Cref{exampleNonUltra} can be generalized to any first-order language $\cL\supsetneq \{<\}$. So
By \L{}os' Theorem, given a first-order language $\cL\supsetneq \{<\}$, there is no first-order theory $T$, such that $\cM\models T\iff \cM$ is o-minimal for every $\cL$-structure $\cM$.

Here we focus our attention on some properties implied by o-minimality which are first-order, i.e., those properties which both hold in all o-minimal structures, and, given a language $\cL=\Set{<,\dots}$, can be axiomatized by a set of $\cL$-sentences. Rigorously, we follow the conventions from \cite{ominimalism}, defined below:

\begin{definition}
	Given a language $\cL=\Set{<,\dots}$, let $T^{omin}_{\cL}$ be the set of all $\cL$-sentences satisfied in every o-minimal $\cL$-structure.
	
	An $\cL$-structure $\cM$ for $\cL=\Set{<,\dots}$ is \emph{pseudo-o-minimal} if $\cM\models T_{\cL}^{omin}$.
\end{definition}
\begin{fact}[\text{\cite[Corollary 10.2]{ominimalism}}]
	An $\cL$-structure $\cM$ for $\cL=\Set{<,\dots}$ is pseudo-o-minimal if and only if $\cM$ is elementarily equivalent to an ultraproduct of o-minimal structures.
\end{fact}

The following two definitions are examples of first-order weakenings of o-minimality.

\begin{definition}
	An expansion of a dense linear order without endpoints $\cM = \Braket{M;<,\dots}$ is \emph{definably complete} if every definable subset of $M$ has a least upper bound.
\end{definition}
\begin{definition}
	An expansion of a dense linear order without endpoints $\cM = \Braket{M;<,\dots}$ is \emph{locally o-minimal} if for any definable subset $A\subseteq M$ and any $a\in M$ there are $b_1,b_2\in M$ such that $b_1<a<b_2$ and  if $I=(b_1,a)$ or $(a,b_2)$ then either $I\subset A$ or $I\cap A=\emptyset$.
\end{definition}
Notice that both definable completeness and local o-minimality, in a given language $\cL$, are axiomatized by first-order schemes which hold in any o-minimal structure. Thus, any pseudo-o-minimal $\cL$-structure is definably complete and locally o-minimal.

Fornasiero, Hieronymi, Miller, Schoutens, Servi and others proved many tameness properties for definably complete and for locally o-minimal structures. (See, e.g,  \cite{Miller,Hieronymi1,For1,Servi,For2,Hieronymi2,ominimalism,ForHier}.) Citing all tameness properties proved in this area will be longer than this paper, so we give two elementary examples by Miller:

\begin{fact}[\text{\cite[Corollary 1.5]{Miller}}]\label{MillerTFAE}
	Let $\cM = \Braket{M;<,\dots}$ be an expansion of a dense linear order without endpoints. Then the following are equivalent:
	\begin{enumerate}
		\item \label{MillerDC} $\cM$ is definably complete.
		\item $\cM$ has the \emph{intermediate value property}, i.e., the image of an interval under a definable continuous map is an interval. 
		\item Intervals in $M$ are definably connected, i.e. for every interval $A\subseteq M$ and every disjoint open definable subsets $U,V\subseteq M$, if $A=(A\cap U) \cup (A\cap V)$, then either $A\cap U=\emptyset$ or $A\cap V=\emptyset$.
		\item $M$ is definably connected.
	\end{enumerate}
\end{fact}
\begin{fact}[\text{\cite[Proposition 1.10]{Miller}}]\label{MillerMapping}
	Let $\cM=\Braket{M;<,\dots}$ be definably complete. Let $f:A\to M^n$ be definable and continuous with $A$ closed and bounded. Then $f(A)$ is closed and bounded. 
	In particular, If $f:A\to M$ is definable and continuous with $A$ closed and bounded, then $f$ achieves a maximum and a minimum on $A$.
\end{fact}

In \cite{ominimalism}, Schoutens presented a strengthening of local o-minimality by the name of \emph{type completeness}, as defined below. In a sense this strengthening extends the locality to $\pm \infty$:
\begin{definition}
	An expansion of a dense linear order without endpoints $\cM = \Braket{M;<,\dots}$ is \emph{type complete} if it is locally o-minimal and, in addition, for any definable subset $A\subseteq M$ there are $c_1,c_2\in M$ such that if $I=(-\infty,c_1)$ or $(c_2,+\infty)$, then either $I\subset A$ or $I\cap A=\emptyset$.
\end{definition}
Type completeness is a first-order scheme, and therefore satisfied by any pseudo-o-minimal structure. 

Several tameness results were proved for definably complete type complete structures in \cite{ominimalism}. For example, a version of o-minimal cell decomposition called \emph{quasi-cell decomposition} (\cite[Theorem 8.10]{ominimalism}) and the following monotonicity theorem:
\begin{fact}[\text{\cite[Theorem 3.2]{ominimalism}}]
	Let $\cM=\Braket{M;<,\dots}$ be a definably complete type complete structure. The set of discontinuities of a one-variable definable
	map $f: Y\to M$ is discrete, closed, and bounded, and consists entirely of jump
	discontinuities. Moreover, there is a definable discrete, closed, bounded subset $D\subseteq Y$
	so that in between any two consecutive points of $D\cup\Set{\pm \infty}$, the map is monotone,
	that is to say, either strictly increasing, strictly decreasing, or constant.
\end{fact}

Of particular importance in the study of definably complete structures are the definable \emph{pseudo-finite} sets, as defined below.
\begin{definition}
	Let $\cM=\Braket{M;<,\dots}$ be a definably complete structure. A definable subset $A\subset M^n$ is \emph{pseudo-finite} if it is closed, bounded, and discrete.
\end{definition}
These definable sets play a role in each of the papers cited above. We follow the convention in \cite{For1, For2}, where there is an extensive study of pseudo-finite sets and their tameness properties. In \cite{For2}, the wording was justified in the definably complete context by saying that pseudo-finite sets are first-order analogue of finite subsets of $\bR^n$, with evidence given by numerous tameness properties of such sets.

One must not confuse \emph{pseudo-finite} sets defined above with \emph{pseudo-$o$-finite} sets, as we define below, coined in \cite{ominimalism}. Though, as we will see in \Cref{pseudo iff pseudoo} the two definitions coincide if $\cM$ is assumed to be pseudo-o-minimal.
\begin{definition}
	Let $\cM=\Braket{M;<,\dots}$ be a pseudo-o-minimal structure. A definable set $X\subseteq M^n$ is \emph{pseudo-$o$-finite} if $(\cM,X)$ satisfies the common theory of o-minimal structures expanded by a unary predicate for a distinguished finite subset.
\end{definition}
The following \namecref{pseudo iff pseudoo} can be immediately extracted from \cite[Corollary 12.6]{ominimalism} together with \cite[Theorem 12.7]{ominimalism}.
\begin{fact}\label{pseudo iff pseudoo}
	Let $\cM=\Braket{M;<,\dots}$ be a pseudo-o-minimal structure. A definable set $A\subset M^n$ is pseudo-finite if and only if it is pseudo-$o$-finite.
\end{fact}
A tameness property of pseudo-finite sets occurring naturally is ``the discrete pigeonhole principle'' \cite{ominimalism}. (Or just ``the pigeonhole principle'' in \cite{For1, For2}.)
\begin{definition}
	An expansion of a dense linear order without endpoints $\cM = \Braket{M;<,\dots}$ has the \emph{pigeonhole principle} if for any pseudo-finite $X\subset K^n$ and definable $f:X\to X$, if f is injective, then it is surjective.
\end{definition}
We remark that the pigeonhole principle can be formulated as ``every pseudo-finite set is definably Dedekind finite'', and as this is a first-order scheme, every pseudo-o-minimal structure has the pigeonhole principle.

In \cite{For1} and \cite{For2}, Fornasiero conjectured the following:
\begin{conjecture}\label{conjecture Fornasiero}
	If $\cK = $ $\Braket{K,+,\cdot,<,\dots}$ is a definably complete expansion of a real closed field, then $\cK$ has the pigeonhole principle.
\end{conjecture}
This conjecture remained open even for $\cK$ a definably complete expansion of a dense linear order. Clearly, the conjecture holds for $\cK$ pseudo-o-minimal. Consequently, it is connected to two other questions asked by Schoutens in \cite{ominimalism}:
\begin{question}\label{DCTCpigeonhole}
	Does every definably complete type complete structure have the pigeonhole principle?
\end{question}
\begin{question}\label{IsThereAnAxiomatization}
	Is there an axiomatization of pseudo-o-minimality by first-order conditions
	on one-variable formulae only?
\end{question}

To clarify the meaning of a \emph{first-order conditions
on one-variable formulae only}, this does not mean a first-order sentence conditioned on a specific one-variable formula, as the following example demonstrates how any first-order theory is axiomatized by such sentences, in particular $T^{omin}$.
\begin{example}
	Let $\cL$ be any language and $T$ be any $\cL$-theory (not necessarily complete). For every sentence $\sigma\in T$, let $\psi_\sigma(x):= x=x\land \sigma$ and let $\phi_\sigma := \exists x\, \psi_\sigma(x)$. Then $\phi_\sigma$ is a first order condition on $\psi_\sigma$, however $\vdash \psi_\sigma\iff \sigma$, so $\Set{\psi_\sigma | \sigma \in T}$ is an axiomatization of $T$.
\end{example}

Clearly, this is not the intended meaning in the question. Rather, following the terminology of \cite{ominimalism}, we interpret a \emph{first-order condition
	on one-variable formulae} as first-order scheme ranging over all one-variable formulae. Rigorously, a first-order condition on one-variable formulae is obtained as follows:
	\begin{itemize}
		\item Let $\tau$ be a first-order sentence in the language $\Set{<,U}$ where $U$ is a unary predicate.
		\item Let $\Phi$ be the set of partitioned $\cL$-formulae $\vphi(x;\bar{y})$ where $x$ is a single variable and $\bar{y}$ is a finite tuple of variables not appearing in $\tau$.
		\item For every $\vphi(x;\bar{y})\in \Phi$, let $\tau_\vphi(x;\bar{y})$ be the $\cL$-formula obtained by replacing any instance of $U(x)$ by $\vphi(x;\bar{y})$.
		\item $T_\tau := \Set{ \forall \bar{y} \,\tau_\vphi(x;\bar{y}) | \vphi(x;\bar{y}) \in \Phi} $.
	\end{itemize} 
	For example, definable completeness is axiomatized in the above fashion by setting $\tau$ to be
	 \begin{align*}
	  & \exists v\, \forall w\, \left(U(w)\rightarrow w<v\right) \rightarrow \\
	 &		  \exists v \left( \forall w\, \left(U(w)\rightarrow w<v\right) \land \forall v'\left(\forall w\, \left(U(w)\rightarrow w<v\right)\right)\rightarrow v\leq v' \right).
	 \end{align*}
		Namely, $\tau$ is the $\Set{<,U}$-sentence stating if $U$ is bounded, then it has a least upper bound.  Following the same terminology, an axiomatization of pseudo-o-minimality by first-order conditions	on one-variable formulae only  is an $\cL$-theory $T'$ such that \[ T^{omin}\supseteq T'\supseteq \cup\Set{T_\tau | \tau \text{ is an $\Set{<,U}$-sentence}}. \]
		In \cite{Rennet}, Rennet showed that there is no \emph{recursive} first-order axiomatization of pseudo-o-minimality in the language of rings $\Set{+,-,\cdot,0,1}$.
		In particular, as definable completeness and type completeness are both recursive first-order schemes, given a recursive language, they cannot axiomatize pseudo-o-minimality.

In this paper, we show a stronger result (with respect to one-variable definable sets) by constructing two ordered structures $\cM, \cN$ on the same universe, in the same language, with the same definable subsets in one variable, where $\cM$ is pseudo-o-minimal and $\cN$ does not have the pigeonhole principle. This gives a negative answer to both \Cref{DCTCpigeonhole,IsThereAnAxiomatization}, as well as a partial answer to \Cref{conjecture Fornasiero}. Furthermore, this gives a stronger result then a negative answer to \Cref{IsThereAnAxiomatization}. It shows that not only is there no first order axiomatization $T'$ as above, but also there is no second order theory in the language $\cL_{\Def}:=\Set{<, \Def}$ where $\Def$ is a unary predicate on subsets interpreted as the definable subsets. This result is strictly stronger as any axiomatization $T'$ as above is equivalent to a second order theory in $\cL_{\Def}$, but not vice-versa.

This also implies that there is no result analogous to \Cref{pseudo iff pseudoo} in the theory of definably complete type complete structures, namely there is a definably complete type complete structure $\cM$ and a pseudo-finite subset $X\subset M$ such that $(\cM,X)$ does not satisfy the common theory of definably complete type complete structures expanded by a unary predicate for a finite set.

It is still open whether we can extend this result to the case where $\cM_0$ is an expansion of a real closed field and fully answer \Cref{conjecture Fornasiero}.

\subsection*{Acknowledgements} The author is grateful to Phillip Hieronymi for presenting the question which motivated this paper, as well as for the fruitful discussions. The author is grateful to Assaf Hasson for the fruitful discussions and the warm support along the way. The author thanks Itay Kaplan for his helpful comments on a preliminary version of this paper and for the discussion which helped clarify and strengthen the main result.

\subsection*{Outline} The construction is done as follows: In \Cref{sectionDefT0T1}, the theory $T_0$ is constructed as an expansion of a dense linear without endpoints by a predicate for a discrete,  closed, and bounded set $Z$ and some extra structure in the language $\cL_0$ such that $T_0\supseteq T_{\cL_0}^{omin}$. We then introduce an expansion $\cL_1\supset\cL_0$ and $T_1\supset T_0$ an $\cL_1$-theory containing a function symbol $f$ which is bijective on $Z$. We show $T_0$ and $T_1$ are consistent. In \Cref{secT0QE} we prove quantifier elimination for $T_0$. 

In \Cref{secT2T1}, we give the construction of $\cM_2$ which will be an expansion of some model $\cM_0$ of $T_0$ to $\cL_1$ with the same one-variable definable sets as $\cM_0$ such that $\cM_2$ does not have the pigeonhole principle. This is done by tweaking a given model $\cM_1$ of $T_1$ expanding $\cM_0$ so that $f$ is now injective but not surjective. It is done carefully enough, so that any definable set in $\cM_2$ differs from a set definable in $\cM_1$ by finitely many constant terms. In \Cref{secT2QE} we show quantifier elimination in $\cM_2$ and deduce that any definable subset of $\cM_2$ is definable in $\cM_0$. We then define $\cM$ to be a trivial expansion of $\cM_0$ to $\cL_0$ and $\cN$ to be $\cM_2$ and show that $\cM,\cN$ possess the properties proclaimed in the introduction.

\section{Preliminaries - cyclic orders}\label{secCyclic}

In this \namecref{secCyclic}, we present the standard definition of a \emph{cyclic order}, as defined below, and present some of its properties needed for the construction following.

\begin{definition}\label{defCyclic}
	A \emph{cyclic order} on a set $A$ is a ternary relation $C$ satisfying the following axioms:
	\begin{enumerate}
		\item Cyclicity: If $C(a,b,c)$, then $C(b,c,a)$.
		\item Asymmetry: If $C(a,b,c)$, then not $C(c,b,a)$.
		\item Transitivity: If $C(a,b,c)$ and $C(a,c,d)$, then $C(a,b,d)$.
		\item Totality: If $a,b,c$ are distinct, then either $C(a,b,c)$ or $C(c,b,a)$.
	\end{enumerate}
\end{definition}

The following fact is folklore (e.g., \cite{Cyc1} and \cite[Part I, \S 4]{Cyc2}) and can be easily verified:
\begin{fact}
	If $\Braket{A,<}$ is a linearly  ordered set, then the relation defined by
	\[ C_<(a,b,c) \iff (a<b<c) \lor (b<c<a) \lor (c<a<b) \]
	is a cyclic order on $A$.
	
	\noindent \emph{We call $C_<$ the \emph{cyclic order induced by $<$}.}
\end{fact}
\begin{definition}
	Let $(X,<)$ be a linearly ordered set. A $<$-cut in $X$ is a pair of subsets $(A,B)$ of $X$ such that $X = A\cupdot B$ and $a<b$ for every $a\in A, b\in B$.
\end{definition}
\begin{fact}[\text{\cite[Lemma 3.8]{CyclicCUTS}}]\label{factCuts}
	Let $X$ be a set with two linear orders, $<_1,<_2$, on $X$. Let $(A_1,B_1)$ be a $<_1$-cut in $X$ and $(A_2,B_2)$ be a $<_2$-cut in $X$. If $(A_1,<_1)\cong (B_2,<_2)$ and $(B_1,<_1)\cong (A_2,<_2)$, then $C_{<_1}=C_{<_2}$.
\end{fact}
\begin{definition}
	Let $C$ be a cyclic order on a set $A$. For any $a,b\in A$, denote
	\[ C(a,-,b):=\Set{x\in A | C(a,x,b)}. \]
\end{definition}
\begin{lemma}\label{cyclicSeperation}
	Let $C$ be a cyclic order on a set $A$ and let $a,b,c\in A$. If $C(a,b,c)$ then
	\[ C(a,-,c) = C(a,-,b)\cup \Set{b}\cup C(b,-,c). \]
\end{lemma}
\begin{proof}
	\begin{itemize}
		\item To prove $C(a,-,c) \supseteq C(a,-,b)\cup \Set{b}\cup C(b,-,c)$:
		\begin{itemize}
			\item By definition, $b\in C(a,-,c)$.
			\item If $C(a,x,b)$, then together with $C(a,b,c)$, and transitivity, we get $C(a,x,c)$.
			\item If $C(b,x,c)$, then by cyclicity, $C(c,b,x)$. By cyclicity again, $C(c,a,b)$. Now by transitivity, $C(c,a,x)$, which is equivalent by cyclicity to $C(a,x,c)$.		 
		\end{itemize}
		\item To prove $C(a,-,c) \subseteq C(a,-,b)\cup \Set{b}\cup C(b,-,c)$, if
		\[x\notin \left(C(a,-,b)\cup \Set{b}\cup C(b,-,c)\right), \]
		then $x\notin \Set{a,b,c}$ and by totality, $C(b,x,a)$ and $C(c,x,b)$. By cyclicity, we get that
		$C(x,a,b)$ and $C(x,b,c)$, which in turn, by transitivity, implies $C(x,a,c)$ which by cyclicity is equivalent to $C(c,x,a)$ which by asymmetry, implies that $x\notin C(a,-,c)$.
	\end{itemize}
\end{proof}
\begin{definition}
	Let $C$ be a cyclic order on a set $A$ and let $X\subseteq A$. Two elements $a,b\in A$ are \emph{$X$-close} if 
	either $X\cap C(a,-,b)$ or  $X\cap C(b,-,a)$ is finite. 
	
	Denote $a\sim_X b$ if $a,b\in A$ are \emph{$X$-close}.
\end{definition}
\begin{lemma}\label{XcloseEq}
	Let $C$ be a cyclic order on a set $A$ and let $X\subseteq A$. Then $\sim_X$ is an equivalence relation on $A$.
\end{lemma}
\begin{proof}
	\begin{itemize}
		\item $X\cap C(a,-,a) =\emptyset$ for all $a\in A$, so reflexivity holds.
		\item Symmetry is obvious by definition.
		\item To prove transitivity, let $a,b,c\in A$ such that $a\sim_X b$ and $b\sim_X c$. Assume towards a contradiction that $X\cap$ $C(a,-,c)$ and $X\cap C(c,-,a)$ are both infinite. We may further assume, without loss of generality, that $X\cap C(a,b,c)$. So by cyclicity, also $X\cap C(c,a,b)$ and $X\cap C(b,c,a)$. By \Cref{cyclicSeperation},
		\begin{align}
		X\cap C(a,-,c) = \left(X\cap C(a,-,b)\right)\cup \left(X\cap \Set{b}\right)\cup \left(X\cap C(b,-,c)\right) \label{XcloseEqqq1}  \\
		X\cap C(c,-,b) = \left(X\cap C(c,-,a)\right)\cup \left(X\cap \Set{a}\right)\cup \left(X\cap C(a,-,b)\right) \label{XcloseEqqq2} \\
		X\cap C(b,-,a) = \left(X\cap C(b,-,c)\right)\cup \left(X\cap \Set{c}\right)\cup \left(X\cap C(c,-,a)\right)\label{XcloseEqqq3} 
		\end{align}
		By Equation (\ref{XcloseEqqq2}), $X\cap C(c,-,b)$ is infinite and by Equation (\ref{XcloseEqqq3}), $X\cap C(b,-,a)$ is infinite. But by Equation (\ref{XcloseEqqq1}), either $X\cap C(a,-,b)$ or $X\cap C(b,-,c)$ is infinite, so either $a\not\sim_X b$ or $b\not\sim_X c$.
	\end{itemize}
\end{proof}
\begin{lemma}\label{cyclicInfPreservesCyclic}
	Let $C$ be a cyclic order on a set $A$. Let $a,a',b,b',c,c'\in A$ and let $X\subseteq A$ such that \begin{align*}
	a\sim_X a',\ b\sim_X b',\ c\sim_X c', \\ a\not\sim_X b,\ a\not\sim_X c,\ b\not\sim_X c.
	\end{align*} Then $C(a,b,c)\iff C(a',b',c')$.
\end{lemma}
\begin{proof}
	By symmetry of $\sim_X$ and cyclicity of $C$, it suffices to show that $C(a,b,c)\implies C(a,b,c')$. So assume towards a contradiction 
	\begin{align}
	C(a,b,c) \label{cyclicInfPreservesCyclic11} \\
	C(c',b,a)\label{cyclicInfPreservesCyclic12}
	\end{align}
	By cyclicity on (\ref{cyclicInfPreservesCyclic12}), we get 
	\begin{align}
	C(a,c',b)\label{cyclicInfPreservesCyclic2c}
	\end{align}
	By transitivity applied to (\ref{cyclicInfPreservesCyclic11}) and  (\ref{cyclicInfPreservesCyclic2c}) we get $C(a,c',c)$, which in turn by cyclicity is equivalent to (\ref{cyclicInfPreservesCyclic4b}) below. By cyclicity and transitivity applied to (\ref{cyclicInfPreservesCyclic11}) and (\ref{cyclicInfPreservesCyclic12}), we get  (\ref{cyclicInfPreservesCyclic4a}) below.
	\begin{align}
	C(c,a,c')\label{cyclicInfPreservesCyclic4b} \\
	C(c',b,c) \label{cyclicInfPreservesCyclic4a} 
	\end{align}
	By the assumption of the \namecref{cyclicInfPreservesCyclic}, either $C(c',-,c)$ or $C(c,-,c')$ is finite. By \Cref{cyclicSeperation} and by (\ref{cyclicInfPreservesCyclic4b}) and  (\ref{cyclicInfPreservesCyclic4a}), this implies that at least one of the following is finite: $C(c',-,b),C(b,-,c),C(c,-,a),C(a,-,c')$, so $a\sim_X c$ or $b\sim_X c$. Contradiction.
\end{proof}

\section{Definitions of $T_0$ and $T_1$}\label{sectionDefT0T1}

\begin{definition}\label{defT0}
	Let $\cL_0:=\Braket{<, Z; S,P,\pi; c_1,c_2,c_3,c_4}$ where $<$ is a binary relation symbol, $Z$ is a unary predicate, $S,P, \pi$ are function symbols  and $c_1,c_2,c_3,c_4$ are constant symbols.
	Let $T_0$ be the $\cL_0$-theory consisting of the following axioms:
	\begin{enumerate}
		\item\label{axiomTomin} $T^{omin}_{\cL_{{0}}}$.
		\item\label{axiomDLO}
		$<$ is a dense linear order without end points.
		\item\label{axiomDiscrete}
		$Z$ is discretely ordered, i.e., every non-maximal (respectively, non-minimal) element in $Z$ has an immediate successor (respectively, predecessor) in $Z$.
		\item\label{axiomClosed} $Z$ is closed, i.e., for all $x\notin Z$, there is an interval disjoint from $Z$ containing $x$.
		\item\label{axiomMinMax} 
		$\min(Z) = c_1, \max(Z)=c_4$.
		\item\label{axiomInfBetween}  $c_2,c_3 \in Z$ are such that $c_1<c_2<c_3<c_4$ and there are infinitely many elements in $Z$ between any two of them.
		\item\label{axiomPi} $\pi$ is the cyclic forward projection on $Z$:
		\[\left(\forall x\right)\ \ \  (\pi(x)\in Z)\land (\forall y\in Z \left(\neg C_<(x,y,\pi(x))\right))\]
		
		\item\label{axiomS} $S$ is defined as the cyclic successor function on $Z$, and as the identity outside of $Z$:
		\[ S(x)=y \iiff \left(x\notin Z \land x=y\right)\lor \left(x\in Z \land y\in Z \land \neg \exists z\in Z\,\left(C_<\left(x,z,y\right)\right)\right) \]
		$P$ is defined as $S^{-1}$.
		\newcounter{asdasd}
		\setcounter{asdasd}{\value{enumi}}
	\end{enumerate}
\end{definition}

The consistency of $T_0$ will be proven together with the consistency of $T_1$ defined in \Cref{defT1} below.

\begin{definition}\label{defT1} 
	Let $\cL_1:=\cL_0\cup \{ f, g \}$ where $f,g$ are unary function symbols.	
	
	Let $T_1$ be $T_0$ together with the following axioms:
	\begin{enumerate}
		\setcounter{enumi}{\value{asdasd}}
		\item \label{axiomg=fInv}\label{g=fInv} $f$ is bijective and $g=f^{-1}$.
		\item \label{fc1c2} $f(Z\cap [c_1,c_2])= Z\cap  [c_3,c_4]$ and $f\upharpoonright\left(Z\cap [c_1,c_2]\right)$ is a partial order isomorphism.
		\item \label{fc2c4} $f(Z\cap (c_2,c_4]) = Z\cap [c_1,c_3)$ and $f\upharpoonright\left(Z\cap (c_2,c_4]\right)$ is a partial order isomorphisms.
		\item\label{finf} For all $n> 1$ and for every $z\in Z$
		\[Z\cap C_<\left(z,-,f^n(z)\right)\text{ and }Z\cap C_<\left(f^n(z),-,z\right)\] are infinite, i.e. $z\not\sim_Z f^n(z)$. 
		
		\noindent \emph{Notice that this is a first-order scheme.}
		
		\item\label{fId} $f(x)=x$ for every $x\notin Z$
		\item\label{circularAxiom} $C_<\left(f^m(z),f^n(z),z\right) $ for all $m>n>0$ and for every $z\in Z$.  
	\end{enumerate}
\end{definition}

\begin{proposition}\label{T0Consistent}
	$T_1$ is consistent.
\end{proposition}
\begin{proof} We prove finite satisfiability of $T_1$ take some sufficiently large natural number $N$. Take $Z=\Set{0,\dots,N}\times \Set{0,\dots,N}$ with the lexicographic order and consider a structure $\cM$ which is a DLO containing $Z$ as an ordered subset.
	
	Let $c_1:= (0,0), c_2 := (0,N), c_3 := (N,0), c_4 :=(N,N)$.
	
	Let \[f((a,b)): = \twopartdef{(a-1 \mod (N+1),b)}{x=(a,b)\in Z }{x}{x\notin Z }\]
	and let $g := f^{-1}$
	
	Let $\pi$ the circular projection, as defined in Axiom \labelcref{axiomPi}.
	
	Let $S$ be the circular successor function, as defined in Axiom \labelcref{axiomS} and let $P:=S^{-1}$
	
	Then $\cM$ satisfies Axioms \labelcref{axiomTomin,axiomDLO,axiomDiscrete,axiomMinMax,axiomClosed,axiomS,axiomPi,axiomg=fInv,fc1c2,fc2c4,fId} by definition.
	As for Axioms \labelcref{axiomInfBetween,finf,circularAxiom}:
	
	Any finite segment of Axiom \labelcref{axiomInfBetween} is contained in the following axiomatization, for a fixed $k\in \bN$:
	
	\begin{enumerate}
		\item[\labelcref*{axiomInfBetween}$_k$.]  $c_2,c_3 \in Z$ are such that $c_1<c_2<c_3<c_4$ and there are at least $k$ elements in $Z$ between any two of them.
		\item[\labelcref*{finf}$_k$.]\label{finfK}  For all $k>n>0$ and for every $z\in Z$.
		\begin{enumerate}
			\item There are at least $k$ elements in $Z\cap C_<\left(z,-,f^n(z)\right)$.
			\item There are at least $k$ elements in $Z\cap C_<\left(f^n(z),-,z\right)$.
		\end{enumerate}
		\item[\labelcref*{circularAxiom}$_k$.]\label{circularAxiomK} $C_<\left(f^m(z),f^n(z),z\right)$ for all $k>m>n>0$.
	\end{enumerate}
	If $N>k$ then $\cM$ satisfies Axioms  \labelcref*{axiomInfBetween}$_k$ and \labelcref*{finf}$_k$, by definition.

	Under the assumption $N>k$, we prove that $\cM$ satisfies Axiom \labelcref*{circularAxiom}$_k$, thus $T_1$ is finitely satisfiable.
	
	For all $(x,y)\in Z$:
	
	\begin{align*}
	f^m((x,y)) = (x-m \mod (N+1),y) \\
	f^n((x,y)) = (x-n \mod (N+1),y)
	\end{align*}
	
	So proving Axiom \labelcref{circularAxiom}$_k$ reduces to proving that for any $x\in \Set{0,\dots,N}$ and $0<n<m<N$ one of the following holds:
	\begin{enumerate}
		\item[(a)]$ x\ominus m < x\ominus n < x$
		\item[(b)]$ x\ominus n < x     < x\ominus m $
		\item[(c)]$ x    < x\ominus m < x\ominus n $
	\end{enumerate} 
	where $\ominus$ is subtraction modulo $N+1$.

	If $m\leq x$\ \ \ \ \ \ \,\ \ then (a) holds.
	
	If $n\leq x<m$ \ then (b) holds.
	
	If $x<n$\ \ \ \ \ \ \,\ \ then (c) holds.
\end{proof}

\section{Quantifier Elimination in $T_0$}\label{secT0QE}

We now show that $T_0$ eliminates quantifiers:

\begin{remark}\label{conseqDLOZ}\ Let $\cM\models T_0$ and $\tau, a\in \cM$. Then the following hold:
	
	\begin{enumerate}
		\item $ S(\tau) \in Z \iff \tau \in Z $
		\item $ P(\tau)\in Z\iff \tau \in Z $
		\item $ S(\tau) = a \iff \tau = P(a) $
		\item $ S(\tau) < a \iff \tau < P(a) $
		\item $ S(\tau) > a \iff \tau > P(a) $
		\item $ \pi(\tau) \in Z \iff c_1\in Z $
		\item $ \pi(\tau) = a \iff \left[\tau\leq c_4 \land a\in Z\land P(a)<\tau\leq a\right]\lor \left[ \tau>c_4 \land c_1=a \right] $
		\item $ \pi (\tau)< a \iff \left[\tau\leq c_4 \land \tau \leq P\circ \pi (a)\right]\lor \left[ \tau>c_4 \land c_1<a \right] $
		\item $ \pi (\tau)> a \iff \left[\tau\leq c_4 \land \tau \geq P\circ \pi \circ S(a)\right]\lor \left[ \tau>c_4 \land c_1>a \right] $

	\end{enumerate}
\end{remark}

\begin{remark}\label{DLOZinZ}
	\ 
	
	If $x\in Z$ then:
	\begin{enumerate}
		\item $S^{m_1}\circ \pi^{n_1} \circ \dots \circ S^{m_{k}}\circ \pi^{n_k}\circ S^{l}(x) = S^{m_1+\dots+m_k+l}(x) $.
		\item $P^{m}(x) = S^{-m}(x)$ and $P^{-m}(x)=S^m(x)$ for all $m\in \bN$.
		\item 
		$ S^{m}(x)\square x \iff S^{m}(c_2)\square c_2$ for all $m\in \bN, \square \in \Set{<,>,=}$, $x\notin \Set{c_4, P(c_4),\dots, P^m(c_4),} $.
		$ P^{m}(x)\square x \iff S^{m}(c_2)\square c_2$ for all $m\in \bN, \square \in \Set{<,>,=}$, $x\notin \Set{c_1, S(c_1),\dots, S^m(c_1),} $.
	\end{enumerate}
	
	If $x\notin Z$:
	\begin{enumerate}
		\item $ S^{m_1}\circ \pi^{n_1} \circ \dots \circ S^{m_{k}}\circ \pi^{n_k}\circ S^{l}(x) = S^{m_1+\dots+m_k}\circ \pi(x) $.
		\item $S^m(x) \square x \iff c_1\square c_1$ for all $m\in \bZ, \square\in \Set{<,>,=}$. 
		\item $S^m\circ \pi (x) = x \iff c_1\neq c_1$ for $m\neq 0$.
		\item $S^m\circ \pi(x)>x\iff S^{m+1}\circ \pi(x)> \pi(x) $.
		\item $S^m\circ \pi(x)<x\iff S^m\circ \pi(x)<\pi(x) $.
	\end{enumerate}
\end{remark}
\begin{lemma}\label{LU}
	For any $\cM\models T_0$ and $a,b\in \cM$, 
	\[ 
	\cM\models \left[\exists x\in Z (a<x<b)\right]\iiff \left[\pi\circ S(a)<b\right] \]
\end{lemma}
\begin{proof}
	If $a\in Z$ then 
	$ \cM\models \left[\exists x\in Z (a<x<b)\right]\iiff \left[S(a)<b\right] $ and $S(a) = \pi\circ S(a)$.
	
	If $a\notin Z$ then 
	$ \cM\models \left[\exists x\in Z (a<x<b)\right]\iiff \left[\pi(a)<b\right] $ and $\pi(a) = \pi\circ S(a)$.
\end{proof}
\begin{proposition}\label{T0QE}
	$T_0$ admits quantifier elimination.
	
	\begin{proof}
		Let $\phi = \exists x \bigwedge_{i\in I} \theta_i\left(\bar{y},x\right) $ such that $\{\theta_i\}_{i\in I}$ are atomic and negated atomic formulas.
		We need to find a quantifier-free $\cL_0$-formula $\varphi$ such \[T_0\models \forall \bar{y}\left[   \left(\exists x \bigwedge_{i\in I} \theta_i\left(\bar{y},x\right) \right)\leftrightarrow \varphi(\bar{y})\right]\]
		
		Firstly, since $\vdash \exists x \big( \chi\left(\bar{y},x\right) \land \theta\left(\bar{y}\right) \big) \leftrightarrow \exists x \big( \chi\left(\bar{y},x\right)\big)  \land \theta\left(\bar{y}\right)$ we may assume that $x$ occurs in $\theta_i$ for all $i\in I$. Secondly, 
		\[\vdash \left[\exists x \bigwedge_{i\in I} \theta_i\left(\bar{y},x\right)\right] \leftrightarrow  \left[\exists x \left( \bigwedge_{i\in I} \theta_i\left(\bar{y},x\right) \land \left( x\in Z \lor x\notin Z  \right)\right)\right]\leftrightarrow\]
		\[\left[ \left(\exists x \left( \bigwedge_{i\in I} \theta_i\left(\bar{y},x\right) \land  x\in Z   \right)\right)\lor \left(\exists x \left( \bigwedge_{i\in I} \theta_i\left(\bar{y},x\right) \land  x\notin Z   \right)\right)\right]  .\]
		So we may assume $\phi$ is either of the form $ \exists x \left( \bigwedge_{i\in I} \theta_i\left(\bar{y},x\right) \land  x\in Z   \right)$ or of the form $\exists x \left( \bigwedge_{i\in I} \theta_i\left(\bar{y},x\right) \land  x\notin Z   \right)$ where $\theta_i$ are atomic and negated atomic formulas such that $x$ occurs in each $\theta_i$. We may assume that  $\theta_i$ is neither `$x\in Z$' nor `$x\notin Z$' for any $i\in I$, as such occurrence would be either superfluous or inconsistent. So each $\theta_i$ is of the form $t_1\square t_2$
		where $t_1,t_2$ are terms with variables in $x,\bar{y}$.
		
		By \Cref{conseqDLOZ}, we may assume either
		\[\phi(\bar{y}) = \exists x \left(\bigwedge_{i=1}^k t_i\square_{i} x \land x\notin Z \right)\]
		or
		\[\phi(\bar{y}) = \exists x \left(\bigwedge_{i=1}^k t_i\square_{i} x \land x\notin Z \right)\]
		where $t_i$ are with variables from $\Set{x,\bar{y}}$, $\square\in \Set{<,>,=,\leq,\geq,\neq}$.
		By \Cref{DLOZinZ}, we may assume that $x$ does not occur in any $t_i$.
		Next, notice that $\geq, \leq, \neq$ are positive Boolean combinations of $<,>,=$ and if $\square_i$ is $``="$ for some $i$ we can just replace $x$ with $t_i$. So we may assume $\square_i\in\Set{<,>}$, i.e. either
		\begin{align}\label{caseInZ}
		\phi(\bar{y}) = \exists x \left(\bigwedge_{i=1}^m l_i <  x \land \bigwedge_{j=1}^n u_i >  x \land x\in Z \right)
		\end{align}
		or
		\begin{align}\label{caseNotInZ}
		\phi(\bar{y}) = \exists x \left(\bigwedge_{i=1}^m l_i <  x \land \bigwedge_{j=1}^n u_i >  x \land x\notin Z \right)
		\end{align}
		where $l_i, u_i$ are terms not containing $x$.

		If $\phi$ is as in (\ref{caseInZ}), then by \Cref{LU}, $\phi(\bar{y})$ is equivalent to
		\[ \bigwedge_{i=1}^n\bigwedge_{j=1}^m \left(\pi\circ S(l_i)<u_j\right). \]
		
		If $\phi$ is as in (\ref{caseNotInZ}), then since $Z$ is co-dense, $\phi(\bar{y})$ is equivalent to
		\[ \bigwedge_{i=1}^n\bigwedge_{j=1}^m \left(l_i<u_j\right). \]
	\end{proof}
\end{proposition}

\section{Definition of $T_2$ and the relation to $T_1$}\label{secT2T1}

\begin{definition}\label{defT2}
	Let $\cM_1\models T_1$ be arbitrary, with universe $M$. 
	
	Let $\cM_0$ be the restriction of $\cM$ to $\cL_0$, i.e., $\cM_0 =\cM_1\upharpoonright\cL_0$. Consequently, $\cM_0\models T_0$.
	
	Let $\cM_2$ be the same $\cL_1$-structure as $\cM$ with a slight modification on $f$ and $g$, as follows.
	
	\begin{align*}
	& f^{\cM_2}(x):=\twopartdef{P\circ f^{\cM_1}(x)}{S^n(x)=c_4\text{ for some }n\in \bN}{f^{\cM_1}(x)}{S^n(x)\neq c_4\text{ for all }n\in \bN}. \\
	& g^{\cM_2}(x):=\twopartdef{g^{\cM_1}\circ S(x)}{S^n(x)=P(c_3)\text{ for some }n\in \bN}{g^{\cM_1}}{S^n(x)\neq P(c_3)\text{ for all }n\in \bN}.\end{align*}

	In words, there is some convex set $X$ with maximum $c_4$ such that the order type of $X\cap Z$ is $\omega^{*}$. $f^{\cM_1}$ maps $X\cap Z$ to a convex subset $f^{\cM_1}(X\cap Z)$ of $Z$ of order type $\omega^*$ with maximum $P(c_3)$, by Axiom \labelcref{fc2c4} in \Cref{defT1}.
	
	Then $f^{\cM_2},g^{\cM_2}$ are obtained from $f^{\cM_1},g^{\cM_1}$ by applying a shift by one element in $X\cap Z$, $f(X\cap Z)$ respectively.
\end{definition}

\begin{lemma}\label{fPreservesCyclic}
	$f^{\cM_1}$ preserves the cyclic order on $Z$, i.e. 
	\[T_1\models \left(\forall z_1,z_2,z_3\in Z\right)\left[C_<\left(z_1,z_2,z_3\right)\iiff C_<\left(f(z_1),f(z_2),f(z_3)\right)\right]. \]
\end{lemma}
\begin{proof}
	Define a new ordering $<'$ on $Z$ by
	\begin{align*}
	x<'y \iff & \left(x,y \in [c_3,c_4] \land x<y\right)\lor \\
	& \left(x,y \in [c_1,c_3) \land x<y\right)\lor \\
	& \left(x\in [c_3,c_4], y \in [c_1,c_3) \right). \\
	\end{align*}
	By Axioms \labelcref{fc1c2,fc2c4} in \Cref{defT1}, $\left([c_1,c_2], < \right) \cong ([c_3,c_4],<')$ and $\left((c_2,c_4], < \right) \cong ([c_1,c_3),<')$ and
	\[T_0\models \left(\forall x,y\in Z\right) \left[x<y \iiff f(x)<'f(y)\right]. \]

	Additionally, by definition of $<'$, it follows that $\left([c_3,c_4],[c_1,c_3) \right)$ is a $<'$-cut in $Z$ and $\left([c_1,c_2], (c_2,c_4] \right)$ is a $<$-cut in $Z$. So by \Cref{factCuts}, $C_{<'} = C_<$. In conclusion
	\begin{align*}
	T_0\models \left(\forall z_1,z_2,z_3\in Z\right)\  \Big[ &C_<\left(z_1,z_2,z_3\right)\iiff \\  & C_{<'}\left(f(z_1),f(z_2),f(z_3)\right) \iiff \\ & C_<\left(f(z_1),f(z_2),f(z_3)\right)\Big].
	\end{align*} 
\end{proof}
\begin{lemma}\label{AbelianInT1} Let $f,g, S, P$ be as in \Cref{defT1}. Then
	$\Braket{f,g,S,P}_{cl}$, the closure of $\Set{f,g,S,P}$ under composition is an Abelian group.
\end{lemma}
\begin{proof}
	By definition, $g\circ f =  I =  P\circ S$, so $f,S$ are invertible and $\Braket{f,g,S,P}_{cl} = \Braket{f,S}_{grp}$ where $\Braket{f,S}_{grp}$ is the group generated by $\Set{f,S}$. 
	
	Since $S$ is definable by the cyclic order on $Z$ (Axiom \labelcref{axiomS} in \Cref{defT0}) and $f$ preserves the cyclic order on $Z$ (\Cref{fPreservesCyclic}), it follows that $f\circ S(x)=S\circ f(x)$. Now $\Braket{f,S}_{grp}$ is Abelian, as the group defined by $\Braket{a,b | ab=ba}$ is Abelian.
\end{proof}
\begin{corollary}\label{ginfLemma} Let $n\geq 1$ and $x\in Z$.
	\[ \left(g^{\cM_1}\right)^n(x)\not\sim_Z x \]
\end{corollary}
\begin{proof}
	Since $x\in Z$, so is $\left(g^{\cM_1}\right)^n(x)$. Therefore, by Axiom \labelcref{finf}, 
	\[ \left(f^{\cM_1}\right)^n\circ \left(g^{\cM_1}\right)^n(x) \not\sim_Z \left(g^{\cM_1}\right)^n(x). \]
	By \Cref{AbelianInT1}, $\left(f^{\cM_1}\right)^n\circ \left(g^{\cM_1}\right)^n(x)=x$.
\end{proof}
\begin{lemma}\label{ff'FinLemma} Let $x\in M$ and $n\in \bN$. There are $k_1,k_2\in \bN$ such that
	\begin{align*}
	\left(f^{\cM_2}\right)^n(x) = P^{k_1}\circ \left(f^{\cM_1}\right)^n(x) \text{ and } \\
	\left(g^{\cM_2}\right)^n(x) = S^{k_2}\circ \left(g^{\cM_1}\right)^n(x)
	\end{align*}
\end{lemma}
\begin{proof}
	By definition of $f^{\cM_2}, g^{\cM_2}$ (\Cref{defT2}), there are $\varepsilon_1,\dots,\varepsilon_n, \nu_1,\dots, \nu_k\in \Set{0,1}$ such that 
	\begin{align}
	& \label{EQQQ1} \left(f^{\cM_2}\right)^n  (x) = P^{\varepsilon_1}\circ f^{\cM_1}\circ \dots \circ P^{\varepsilon_n}\circ f^{\cM_1}(x) \\ 
	& \label{EQQQ2} \left(g^{\cM_2}\right)^n  (x) = S^{\nu_1}\circ g^{\cM_1}\circ \dots \circ S^{\nu_n}\circ g^{\cM_1}(x)
	\end{align}
	By \Cref{AbelianInT1}, the right hand side in Equation (\ref{EQQQ1}) is equal to \[P^{\varepsilon_1+\dots+\varepsilon_n}\circ \left(f^{\cM_1}\right)^n(x)\] and the right hand side in Equation (\ref{EQQQ2}) is equal to \[S^{\nu_1+\dots+\nu_n}\circ \left(f^{\cM_1}\right)^n(x).\]
\end{proof}

\begin{corollary}\label{ff'Fin} For all $n\in \bN$ and every $x\in M$:
	\[\left(f^{\cM_1}\right)^n(x)\sim_Z \left(f^{\cM_2}\right)^n(x) \text{ and } \left(g^{\cM_1}\right)^n(x)\sim_Z \left(g^{\cM_2}\right)^n(x)\] 
\end{corollary}

\section{Quantifier elimination in $T_2$}\label{secT2QE}

In this \namecref{secT2QE}, unless otherwise specified, we work inside $\cM_2$, so $f$ is $f^{\cM_2}$ and $g$ is $g^{\cM_2}$.

\begin{lemma}\label{f'}
	$\cM_2$ satisfies the following:
	\begin{enumerate}
		\item\label{f'c1c2} $f(Z\cap [c_1,c_2])= Z\cap  [c_3,c_4]$ and $f\upharpoonright\left(Z\cap [c_1,c_2]\right)$ is a partial order isomorphism, and its inverse is $g\upharpoonright  Z\cap  [c_3,c_4]$.
		\item\label{f'c2c4} $f(Z\cap (c_2,c_4]) = Z\cap [c_1,P(c_3))$ and $f\upharpoonright\left(Z\cap (c_2,c_4]\right)$ is a partial order isomorphisms, and its inverse is $g\upharpoonright  Z\cap  [c_1,P(c_3))$.
		\item\label{f'Id} $g(x) = f(x)=x$ for every $x\notin Z$
		\item\label{M2axiomFinjective} 	$f$ is injective and not surjective on $Z$. Moreover, $f(Z)= Z\setminus \Set{P(c_3)}$.
		\item\label{M2axiomGF} $g\circ f(x)=x$ for all $x\in M$.
		\item\label{M2axiomFG} $f\circ g(x)=x$ for all $x\in M\setminus \Set{P(c_3)}$.
		\item\label{f'inf}  For all $n\geq 1$ and for every $z\in Z$
		\[Z\cap C_<\left(z,-,f^n(z)\right)\text{ and }Z\cap C_<\left(f^n(z),-,z\right)\] are infinite, i.e. $z\not\sim_Z f^n(z)$. 
		\item\label{g'inf}  For all $n\geq 1$ and for every $z\in Z$
		\[Z\cap C_<\left(z,-,g^n(z)\right)\text{ and }Z\cap C_<\left(g^n(z),-,z\right)\] are infinite, i.e. $z\not\sim_Z g^n(z)$. 
	\end{enumerate}
\end{lemma}
\begin{proof}
	\begin{itemize}
		\item 	\Cref{f'c1c2,f'c2c4,f'Id} follow by definition of $f^{\cM_2}$ and by Axioms \labelcref{fc1c2,fc2c4,fId} in \Cref{defT1}.
		
		\item \Cref{M2axiomFinjective} follows from \Cref{f'c1c2,f'c2c4}, as 
		\begin{align*}
		& Z = \left(Z\cap [c_1,c_2]\right)\cupdot \left(Z\cap(c_2,c_4]\right) \\ 
		& Z\setminus P(c_3) =  \left(Z\cap [c_1,P(c_3))\right)\cupdot \left(Z\cap[c_3,c_4]\right).
		\end{align*}
		\item To prove \Cref{M2axiomGF}, we separate into two cases: 
		\begin{itemize}
			\item if $S^n(x)=c_4$ for some $n\in \bN$, then $S^n\circ f^{\cM_1}(x) =P(c_3)$, so
			\[S^{n+1}\circ f^{\cM_2}(x)= S^{n} \circ S\circ P\circ f^{\cM_1}(x) = S^{n} \circ f^{\cM_1}(x) = P(c_3).\] 
			So by definition of $g^{\cM_2}$,
			\[ g^{\cM_2}\circ f^{\cM_2}(x) = g^{\cM_1}\circ S\circ P\circ f^{\cM_1}(x) = x. \]
			
			\item if $S^n(x)\neq c_4$ for all $n\in \bN$, then $S^n\circ f^{\cM_1}(x)\neq P(c_3)$ for all $n\in\bN$. So
			\[ g^{\cM_2}\circ f^{\cM_2}(x) = g^{\cM_1}\circ f^{\cM_1} (x) = x. \]
		\end{itemize}
		\item To prove \Cref{M2axiomFG}, by \Cref{M2axiomFinjective,f'Id}, for all $x\in M\setminus \Set{P(c_3)}$, $x = f(y)$ for some $y\in M$, therefore by \Cref{M2axiomGF}
		\[ f\circ g(x) = f\circ g\circ f(y) = f(y) = x. \]
		\item \Cref{f'inf} follows from Axiom \labelcref{finf} in \Cref{defT1} and \Cref{ff'Fin}.
		\item \Cref{g'inf} follows from \labelcref{ginfLemma} and \Cref{ff'Fin}.
	\end{itemize}
\end{proof}
\begin{corollary}\label{fgIFF}
	Let $a,b\in Z$, $\square\in \Set{<,>,=}$
	\begin{enumerate}
		\item  If $a\in [c_1,c_2]$ and $b\in [c_3, c_4]$, then $\cM_2\models f(a)\square b \iff a\square g(b)$.
		\item  If $a\in [c_1,c_2]$ and $b\notin [c_3, c_4]$, then $\cM_2\models f(a)\square b \iff c_3\square b$.
		\item  If $a\in (c_2,c_4]$ and $b\in [c_1, P(c_3))$, then $\cM_2\models f(a)\square b \iff a\square g(b)$.
		\item  If $a\in (c_2,c_4]$ and $b\notin [c_1, P(c_3))$, then $\cM_2\models f(a)\square b \iff c_1\square b$.
		\item  If $a\in [c_1,P(c_3))$ and $b\notin (c_2, c_4]$, then $\cM_2\models g(a)\square b \iff c_4\square b$.
		\item  If $a\in [c_3,c_4]$ and $b\notin [c_1, c_2]$, then $\cM_2\models g(a)\square b \iff c_1\square b$.
	\end{enumerate}
\end{corollary}
\begin{proof}\ 
	
	$(1)$ and $(2)$ follow from  \Cref{f'}, \Cref{f'c1c2,M2axiomFG}.
	
	$(3)$ and $(4)$ follow from \Cref{f'}, \Cref{f'c2c4,M2axiomFG}.
	
	$(5)$ follows from \Cref{f'}, \Cref{f'c2c4,M2axiomGF}.
	
	$(6)$ follows from \Cref{f'}, \Cref{f'c1c2,M2axiomGF}.
\end{proof}
\begin{corollary}\label{elimination 1}
	Let $x\in M, y\in Z$, $\square\in \Set{<,>,=}$.
	\begin{align*}
	& \cM_2\models f(x)\square y \iiff 
	\begin{pmatrix}
	\Big(& \left(x\notin Z \right) & \land & x\square y & \Big) & \lor \\
	\Big(& \left(x\in Z \cap [c_1, c_2] \land y\in [c_3,c_4]\right) & \land & x\square g(y) & \Big) & \lor \\
	\Big(& \left(x\in Z \cap [c_1, c_2] \land y\notin [c_3,c_4]\right) & \land & c_3\square y & \Big) & \lor \\
	\Big(& \left(x\in Z \cap (c_2, c_4] \land y\in [c_1, P(c_3))\right) & \land & x\square g(y) & \Big) & \lor \\
	\Big(& \left(x\in Z \cap (c_2, c_4] \land y\notin [c_1, P(c_3))\right) & \land & c_1\square y & \Big) &  
	\end{pmatrix} \\
	& \cM_2\models g(x)\square y \iiff 
	\begin{pmatrix}
	\Big(& \left(x\notin Z \right) & \land & x\square y & \Big) & \lor \\
	\Big(& \left(x\in Z\cap [c_3,c_4]\land y\in [c_1, c_2]\right) & \land & x\square f(y) & \Big) & \lor \\
	\Big(& \left(x\in Z\cap [c_3,c_4]\land y\notin [c_1, c_2]\right) & \land & c_1\square y & \Big) & \lor \\
	\Big(& \left(x\in Z\cap [c_1,P(c_3))\land y\in (c_2, c_4]\right) & \land & x\square f(y) & \Big) & \lor \\
	\Big(& \left(x\in Z\cap [c_1,P(c_3))\land y\notin (c_2, c_4]\right) & \land & c_4\square y & \Big) & \lor \\
	\Big(& \left(x=P(c_3)\right) & \land & P(c_3)\square y & \Big) &	  
	\end{pmatrix}.
	\end{align*}
\end{corollary}

\begin{remark}\label{fgIffTrivial}
	If $x\notin Z$ then $\cM_2\models f(x)=g(x)=x$. In particular,
	\begin{itemize}
		\item $\cM_2\models f(x)\in Z\iiff x\in Z$ for all $x\in M$.
		\item $\cM_2\models g(x)\in Z\iiff x\in Z$ for all $x\in M$.
		\item $\cM_2\models f(x)\square y \iiff g(x)\square y\iiff x\square y $
		for any $x\in M\setminus Z, y\in M$, $\square\in \Set{<,>,=}$.
	\end{itemize}
\end{remark}
\begin{remark}\label{reduce to yin Z Lemma}
	If $x\in Z, y\notin Z$ then:
	\begin{itemize}
		\item $\cM_2\models x>y \iiff x\geq \pi(y)$.
		\item $\cM_2\models x< y \iiff x\leq P\circ \pi(y)$.
	\end{itemize}
\end{remark}
\begin{corollary}\label{reduce to yin Z}
	\begin{align*}
	& T_2\models \left[ x\in Z\land y\notin Z \land x>y \right]\iiff \left[ x\in Z\land y\notin Z \land x\geq \pi(y)\land \pi(y)\in Z \right] \\
	& T_2\models \left[ x\in Z\land y\notin Z \land x<y \right]\iiff \left[ x\in Z\land y\notin Z \land x\leq P\circ\pi(y)\land P\circ\pi(y)\in Z \right] \\
	& T_2\models \left[ x\in Z\land y\notin Z \land x=y \right]\iiff \left[ x\in Z\land y\notin Z \land c_1=y\right]
	\end{align*}
\end{corollary}

\begin{definition}
	Following standard terminology, a \emph{constant term} is a term with no free variables.
\end{definition}

\begin{definition}
	Given two $\cL_1$-definable maps $F,G:M\to M$, denote $F\approx G$ if there are finitely many constant terms $\tau_1,\dots, \tau_k$, such that \[ T_2\models (\forall x)\ \left[F(x)=G(x)\lor \bigvee_{i=1}^k x=\tau_i\right]. \]
	$\approx$ is an equivalence relation. For any $\cL_1$-definable map $F:M\to M$, let $[F]$ be its equivalence class.
\end{definition}

\begin{lemma}\label{fS=Sf}
	$f\circ S\approx S\circ f$.
\end{lemma}
\begin{proof}\ 
	
	\begin{itemize}
		\item If $x\notin Z$  then both $S$ and $f$ are the identity on $x$, so the equality $f\circ S(x)= S\circ f(x)$ is trivial.
		\item If $x\in Z$ and  $c_1< x<c_2$ or $c_2<x<c_4$ then the equality $f\circ S(x)= S\circ f(x)$ follows by \Cref{f'c1c2,f'c2c4} in \Cref{f'}.
	\end{itemize}
	In conclusion, the equality $f\circ S(x)= S\circ f(x)$ holds for all $x\neq c_1, c_2,c_4$.
	
\end{proof}

For any finite-to-one map $F,F',G,G':M\to M$, if $F\approx F'$ and $G\approx G'$ then $F\circ G\approx F\circ G$. Since $f,S,P$ are injective and $g$ is injective outside $\{P(c_3)\}$, the composition $[F]\circ[G]:=[F\circ G]$ is well defined, for any composition of $f,g,S,P$.
\begin{proposition}\label{fgSPAbelian}
	$\Braket{[f],[g],[S],[P]}_{cl}$, the closure of $\Set{[f],[g],[S],[P]}$ under composition is an Abelian group.
\end{proposition}
\begin{proof}
	\begin{align*}
	T_2\supset T_0 &\models P\circ S(x) = S\circ P = x \\
	T_2 &\models \forall (x\neq P(c_3)) g\circ f(x) =f\circ g (x) = x.
	\end{align*}
	So $[g][f]=[f][g]=[P][S]=[S][P]=1$. In particular $[f],[S]$ are invertible and $\Braket{[f],[g],[S],[P]} = \Braket{[f],[S]}_{grp}$ where $\Braket{[f],[S]}_{grp}$ is the group generated by $\Set{[f],[g]}$. By \Cref{fS=Sf}, $[f][S]=[S][f]$. The claim now follows from the fact the group defined by $\Braket{a,b | ab=ba}$ is Abelian.
\end{proof}

\begin{remark}\label{finiteInZstep}
	Let $x\in M$ and  $F\in\Set{S,P,\pi}$. If there are infinitely many elements in $Z$ between $x$ and $F(x)$, then $F(x)\in \Set{c_1,c_4}$.
\end{remark}
By \emph{infinitely many elements in $Z$ between $x$ and $F(x)$}, we mean with respect to the order $<$ and not the cyclic order $C_<$, i.e.,
either $x<F(x)$ and $Z\cap [x,F(x)]\geq \aleph_0$, or $F(x)<x$ and $Z\cap [F(x),x]\geq\aleph_0$.

This is weaker than $x\not\sim_Z F(x)$; for example, $c_1\sim_Z c_4$ but there are infinitely many elements in $Z$ between $c_1$ and $c_4$.
\begin{lemma}\label{finiteInZ}
	Let $F,G\in \Braket{S,P,\pi}_{cl}$. Then there are finitely many constant terms $\tau_1,\dots, \tau_k$, such that if $F(x)\notin\Set{\tau_1,\dots, \tau_k}$, then there are only finitely many elements in $Z$ between $x$ and $F(x)$.
\end{lemma}
\begin{proof}
	Let $F=G_k\circ\dots\circ G_1$ where $G_1,\dots, G_k\in \Set{S,P,\pi}$. Let $F_i:=Id$, $F_i:= G_i\circ\dots\circ G_1$ for any $1\leq i\leq k$, so $F=F_k$. If there are infinitely many elements in $Z$ between $x$ and $F(x)$, then there is some $1\leq i\leq k$ with infinitely many elements in $Z$ between $F_i(x)$ and $F_{i-1}(x)$, so by \Cref{finiteInZstep}, $F_i(x)\in \Set{c_1,c_4}$ and thus $F(x)= F_k(x) = F_{k-i}\circ F_i(x) \in \Set{F_{k-i}(c_1), F_{k-i}(c_4)}$. So if 
	\[F(x)\notin \Set{F_i(c) | 0\leq i\leq k-1, c\in\Set{c_1,c_4} }\] then there are finitely many elements in $Z$ between $x$ and $F(x)$.
\end{proof}
\begin{lemma}\label{f'circularAxiom} 
	$C_<\left(f^m(z),f^n(z),z\right) $ for all $m>n>0$ and for every $z\in Z$.
\end{lemma}
\begin{proof} Let $m>n>0$ and $z\in Z$.
	By Axiom \labelcref{circularAxiom} in \Cref{defT1}, 
	\[C_<\left(\left(f^{\cM_1}\right)^m(z),\left(f^{\cM_1}\right)^n(z),z\right).\]
	By \Cref{ff'Fin}, 	\[\left(f^{\cM_2}\right)^n(z)\sim_Z\left(f^{\cM_1}\right)^n(z)\text{ and }\left(f^{\cM_2}\right)^m(z)\sim_Z\left(f^{\cM_1}\right)^m(z).\]
	and the \namecref{f'circularAxiom} follows from \Cref{cyclicInfPreservesCyclic}.
\end{proof}

\begin{lemma}\label{onetype}
	for any $n\in \bN$ and $z\in Z$:
	\[\cM_2\models f^{n+1}(z)<z\iiff\bigwedge_{i=0}^n ( f^{i}(z)> c_2).\]
\end{lemma}

\begin{proof} We prove the \namecref{onetype} by induction on $n$. For $n=0$ the claim holds by definition of $f$.
	For $n\geq 1$, By \Cref{f'circularAxiom}, $C_<(f^{n+1}(z),f^n(z),z)$. So 
	\[\cM_2\models f^{n+1}(z)<z \iiff f^{n+1}(z)<f^n(z)<z. \]
	By the induction hypothesis, $f^n(z)<z$ is equivalent to $\bigwedge_{i=0}^{n-1} ( f^{i}(z)> c_2)$ and $f^{n+1}(z)<f^n(z)$ is equivalent to $f^n(z)>c_2$.
\end{proof}

\begin{definition}\ 
	\begin{enumerate}
		\item $ \Phi:=\Set{\phi^n |\phi\in\Set{f,g},\ n\in \bN}$.
		\item $\Sigma:=\Set{ \sigma^m |  \sigma\in\set{S,P}, m\in \bN }$.
		\item  $\Pi: = \Set{\pi^{\epsilon} |  \epsilon\in\set{0,1} }$. 
		\item For any functions $h_1,\dots, h_n$ and $A,B\subseteq \Braket{h_1,\dots, h_n}_{cl}$, let $AB:=\Set{a\circ b | a\in A, b\in B}$.
	\end{enumerate}
\end{definition}
\begin{lemma}\label{f is much more}
	Let $n\geq 1$, $\psi_1,\psi_2\in \Sigma\Pi$, and $\square\in \Set{<,>,=}$. Then 
	\begin{enumerate}
		\item There are constant terms $\tau_1,\dots, \tau_k$ such that
		\[T_2\models f^n\circ\psi_1(x)\square \psi_2(x)\iiff \begin{bmatrix}
		\left(  \psi_1(x)\notin Z\land \psi_1(x)\square\psi_2(x) \right) & \lor \\
		\left(  \psi_1(x)\in Z, \psi_1(x),\psi_2(x)\notin \Set{\tau_1,\dots, \tau_k} \land  f^n\circ\psi_1(x)\square\psi_1(x) \right) & \lor \\
		\left(  \bigvee_{i=1}^k \left(\psi_1(x)=\tau_i \land  f^n(\tau_i)\square\psi_2(x)\right) \right) & \lor \\ 
		\left(  \bigvee_{i=1}^k \left(\psi_2(x)=\tau_i \land  f^n\circ\psi_1(x)\square\tau_i\right) \right)  \\ 
		\end{bmatrix}. \]
		\item There are constant terms $\sigma_1,\dots, \sigma_l$ such that
		\[T_2\models g^n\circ\psi_1(x)\square \psi_2(x)\iiff \begin{bmatrix}
		\left(  \psi_1(x)\notin Z\land \psi_1(x)\square\psi_2(x) \right) & \lor \\
		\left(  \psi_1(x),\psi_2(x)\in Z\setminus \Set{\sigma_1,\dots, \sigma_l} \land  g^n\circ\psi_1(x)\square\psi_1(x) \right) & \lor \\
		\left(  \bigvee_{i=1}^l \left(\psi_1(x)=\sigma_i \land  g^n(\sigma_i)\square\psi_2(x)\right) \right) & \lor \\ 
		\left(  \bigvee_{i=1}^l \left(\psi_2(x)=\sigma_i \land  g^n\circ\psi_1(x)\square\sigma_i\right) \right)  \\ 
		\end{bmatrix}. \]
	\end{enumerate}
\end{lemma}
\begin{proof} 
	\begin{enumerate}
		\item By \Cref{finiteInZ} applied twice, there are constant terms $\tau_1,\dots,\tau_k$ such that whenever $\psi(x)_1,\psi_2(x)\notin \Set{\tau_1,\dots, \tau_k}$, there are finitely many elements in $Z$ between $\psi(x)_1$ and $\psi_2(x)$.
		\begin{itemize}
			\item 	If $\psi_1(x)\notin Z$, then by \Cref{f'Id} of \Cref{f'}, $f^n\circ\psi_1(x)=\psi_1(x)$. In particular, 
			\[\cM_2\models  f^n\circ \psi_1(x)\square \psi_2(x) \iiff  \psi_1(x)\square \psi_2(x).   \]
			\item 	If $\psi_1(x)\in Z$, $\psi_1(x),\psi_2(x)\notin \Set{\tau_1,\dots, \tau_k}$, then by \Cref{f'}, \Cref{f'inf} there are infinitely many elements in $Z$ between $\psi_1(x)$ and $f^n\circ\psi_1(x)$. As there are only finitely many elements in $Z$ between $\psi_1(x)$ and $\psi_2(x)$, it follows that
			\[\cM_2\models  f^n\circ \psi_1(x)\square \psi_2(x) \iiff f^n\circ \psi_1(x)\square \psi_1(x).   \]
			
		\end{itemize}
		
		\item The proof is similar.
	\end{enumerate}
	
\end{proof}
\begin{definition}\ 
	\begin{enumerate}
		\item
		We define $\deg(F)$ for $F\in \Braket{f,g,S,P,\pi}_{cl}$ inductively, as follows:
		\begin{itemize}
			\item $\deg(Id)=\deg(S)=\deg(P)=\deg(\pi) = 0$
			\item $\deg(f)=\deg(g) = 1$.
			\item $\deg(F\circ G) = \deg(F)+\deg(G)$ for all $F,G\in \Braket{f,g,S,P,\pi}_{cl}$.
		\end{itemize}
		
		\emph{Notice that this is a syntactic definition, e.g., $\deg(F\circ G) = 2$.}
		
		\item  For any quantifier free $\cL_1$-formula $\theta(x,\bar{y})$ and variable $x$ we define $\rank(\theta, x)\in \left(\Set{-\infty}\cup\bN\right)^2$ by induction on the complexity of $\theta$:
		\begin{itemize}
			\item If $x$ does not occur in $\theta$, then $\rank(\theta,x) = (-\infty, -\infty)$.
			\item If $\theta$ is atomic of the form $F(x)\in Z$  then $\rank(\theta,x) = (-\infty, \deg(F))$.
			\item If $\theta$ is atomic of the form $F(x)\square \tau$ where $F\in \Braket{f,g,S,P}_{cl}$, $\square\in \Set{<,>,=}$, and $\tau$ is an $\cL_1$-term such that $x$ does not occur in $\tau$, then $\rank(\theta,x) = (-\infty, \deg(F))$.
			\item  If $\theta$ is atomic of the form $F(x)\square G(x)$ where $F, G\in \Braket{f,g,S,P}_{cl}$, $\square\in \Set{<,>,=}$, and $\deg(F)\leq \deg(G)$, then $\rank(\theta,x) = (\deg(F), \deg(G))$.
			\item If $\theta$ is a Boolean combination of atomic formulas $\theta_1,\dots, \theta_k$, then
			$\rank(\theta,x)$ is the lexicographic maximum of $\Set{\rank(\theta_i,x)}_{i=1}^k$.
		\end{itemize} 
	\end{enumerate}
\end{definition}

\begin{definition}\ 
	A quantifier free $\cL_1$-formula $\theta(x,\bar{y})$ is \emph{$x$-corrected} if any term $F(x)$ appearing in $\theta$ belongs to $\Phi\Sigma\Pi$.
\end{definition}

\begin{lemma}\label{reduceToFSP} 
	For any quantifier free $\cL_1$-formula $\vphi$ and variable $x$, there is some $x$-corrected formula $\phi$ such that $\rank(\vphi,x)\leq\rank(\phi,x)$ and $T_2\models \vphi\iiff\phi$.
\end{lemma}
\begin{proof}
	A Boolean combination of $x$-corrected formulas is $x$-corrected, so we may assume $\vphi$ is atomic. 
	
	\begin{itemize}
		\item If $\vphi$ is of the form $F(x)\in Z$ for some $G\in \Braket{f,g,S,P,\pi}_{cl}\setminus \Braket{f,g,S,P}_{cl}$ then $T_2\models F(x)\in Z \iiff \pi(x)\in Z$.
		\item If $\vphi$ is of the form $F(x)\in Z$ for some $G\in  \Braket{f,g,S,P}$ then $T_2\models G(x)\in Z \iiff x\in Z$.
		
		\item If $\vphi$ of the form $F(x)\square \tau$ for some term $\tau$ and $\square\in\Set{<,>,=}$:
		
		\begin{itemize}
			\item If  $F\in \Braket{f,g,S,P}_{cl}$, then by \Cref{fgSPAbelian}, there is some $F'\in \Phi\Sigma$ with $\deg(F')=\deg(F)$, and constant terms $\tau_1,\dots, \tau_k$ such that $F(x)=F'(x)$f for all $x\notin \Set{\tau_1,\dots,\tau_k}$.
			
			So \[\begin{matrix}
			T_2\models & \left[F(x)\square\tau\right]\iiff
			& \begin{bmatrix}
			\left(\bigwedge_{i=1}^k x\neq \tau_k \land F'(x)\square \tau\right) & \lor \\
			\bigvee_{i=1}^k \left( x=\tau_i \land F(\tau_i)\square \tau \right)
			\end{bmatrix}
			\end{matrix}. \]
			
			\item If $F\in \Braket{f,g,S,P,\pi}_{cl}\setminus \Braket{f,g,S,P}_{cl}$, then there are $F_1, F_2\in \Braket{f,g,S,P}_{cl}$ such that $\cM_2\models F(x) = F_1\circ\pi\circ F_2(x)$ for all $x\in M$ and $\deg(F_1)\leq \deg(F_1\circ F_2)=\deg(F)$. So $\cM_2\models F(x)=F_1\circ F_2(x)$ for all $x\in Z$ and $\cM_2\models F(x)=F_1\circ\pi (x)$ for all $x\notin Z$. 
			So
			
			\[\begin{matrix}
			T_2\models & \left[F(x)\square \tau\right]\iiff  &
			\begin{bmatrix}
			\left(x\in Z \land F_1\circ F_2(x)\square \tau\right) \lor \\ 
			\left(x\notin Z \land F_1\circ\pi(x)\square \tau\right) \\ 
			\end{bmatrix} 
			\end{matrix}\]
			and $F_1,F_1\circ F_2\in \Braket{f,g,S,P}_{cl}$, so we can apply the previous case to get a formula where every term $F(x)$ to the left of $\square$ belongs to $\Phi\Sigma\Pi$.

		\end{itemize}
		If $x$ does not appear in $\tau$ we are done. Otherwise, if $\tau=G(x)$ for some term $G$, a symmetric argument applied to $G$ will an $x$-corrected formula $\phi$ equivalent to $\vphi$ as needed.
	\end{itemize}
\end{proof}

\begin{lemma}\label{rank infty n+1}\label{rank n+1 k}
	\begin{enumerate}
		\item 	Let $\vphi$ be an $x$-corrected atomic formula of rank $(-\infty, n+1)$ or of rank $(n+1,k)$ for some $n, k\in \bN$. Then there is some quantifier free formula $\phi$ such that $\rank(\phi,x)<\rank(\vphi,x)$ and $T_2\models \vphi\iiff\phi$.
	\end{enumerate}
\end{lemma}
\begin{proof}
	\begin{enumerate}[label=(\arabic*)]
		\item\label{rank n+1 k ITEM} Assume $\rank(\vphi,x) = (-\infty, n+1)$.
		
		\textbf{If $\vphi$ is of the form $F\circ H(x)\in Z$} where $F\in \Phi, H\in \Sigma\Pi$,
		by \Cref{fgIffTrivial}, $\cM_2 \models F\circ H(x)\in Z\iff H(x)\in Z$ and $\rank(H(x)\in Z,x)=(-\infty, 0)$.
		\bigskip
		
		\textbf{If $\vphi$ is of the form  $F\circ H(x)\square \tau$} where $F\in \Phi, H\in \Phi\Sigma\Pi$, $\deg(F)=1,\deg(H)=n$, and $\tau$ is some $\cL_1$-term not containing $x$
		
		In which case, $F\in \Set{f,g}$ and
		\[\cM_2\models \left[F\circ H(x)\square \tau \right] \iiff 
		\begin{bmatrix}
		\left(F\circ H(x)\square \tau \land \tau\in Z \right)& \lor \\
		\left(F\circ H(x)\square \tau \land \tau\notin Z\land H(x)\notin Z\right) & \lor \\
		\left(F\circ H(x)\square \tau \land \tau\notin Z\land H(x)\in Z\right)
		\end{bmatrix}_. \]
		\begin{enumerate}
			\item\label{rank infty n+1 asdasd} Applying \Cref{elimination 1} to $H(x),\tau$, there is some $\phi'(x,\tau)$ $x$-corrected formula $\phi'(x,\tau)$ of rank $(-\infty, n)$ such that $\cM_2\models \phi'(x,\tau)\iiff F\circ H(x)\square \tau$ for all $\tau\in Z$. So 
			\[ \cM_2\models \phi'(x,\tau)\land \tau\in Z \iiff F\circ H(x)\square \tau \land  \tau\in Z \]
			
			\item Applying \Cref{reduce to yin Z} to $F\circ H(x)$ and $\tau$, we obtain that 
			\begin{align}\label{rank infty n+1 asdasdasdf}
			\left(F\circ H(x)\square \tau \land \tau\notin Z\land H(x)\in Z\right)
			\end{align}
			is equivalent to one of the following:
			\begin{align}
			& F\circ H(x) \in Z\land \tau \notin Z \land F\circ H(x)\geq \pi(\tau)\land \pi(\tau)\in Z  \label{rank infty n+1 blaEQ1} \\
			&  F\circ H(x)\in Z\land \tau\notin Z \land F\circ H(x)\leq P\circ\pi(\tau)\land P\circ\pi(\tau)\in Z \label{rank infty n+1 blaEQ2} \\
			& F\circ H(x)\in Z\land \tau\notin Z \land c_1=\tau
			\end{align}
			As in \labelcref{rank infty n+1 asdasd}, there are $\phi'_1(x,\pi(\tau))$ and $\phi'_2(x,P\circ\pi(\tau))$ of rank $(-\infty, n)$ equivalent to $F\circ H(x)\geq \pi(\tau)\land \pi(\tau)\in Z$ and $F\circ H(x)\leq P\circ\pi(\tau)\land P\circ\pi(\tau)\in Z$, respectively. So (\ref{rank infty n+1 asdasdasdf}) is equivalent to one of the following:
			\begin{align}
			&  H(x) \in Z\land \tau \notin Z \land \phi'_1(x,\pi(\tau)) \land  \pi(\tau)\in Z \label{rank infty n+1 blaEQ1A} \\
			&   H(x)\in Z\land \tau\notin Z \land \phi'_2(x,P\circ\pi(\tau)) \land  P\circ\pi(\tau)\in Z\label{rank infty n+1 blaEQ2A} \\
			&  H(x)\in Z\land \tau\notin Z \land c_1=\tau
			\end{align}
		\end{enumerate}
		an each is $x$-corrected of rank $(-\infty, n)$.

		By \Cref{fgIffTrivial}, $\left(F\circ H(x)\square \tau \land \tau\notin Z\land H(x)\notin Z\right)$ is equivalent to \[\left( H(x)\square \tau \land \tau\notin Z\land H(x)\notin Z\right)\] and the latter is an $x$-corrected formula of rank $(-\infty, n)$.
		
		\item Assume $\rank(\vphi,x) = (n+1, k)$. Then
		$\vphi$ is of the form  $F\circ H(x)\square G(x)$ where $F\in \Phi, H, G\in \Phi\Sigma\Pi$, $\deg(F)=1,\deg(H)=n, deg(G)=k$ and $n<k$. Replace $G(x)$ with $G(y)$ on the left of $\square$ to get $F\circ H(x)\square G(y)$. Apply \Cref{rank n+1 k ITEM} of this \namecref{rank infty n+1} to $F\circ H(x)\square G(y)$ and get some quantifier-free formula $\phi'(x,y)$ with $\rank(\phi'(x,y), x) \leq (-\infty,n)$. Replacing back, we get $\rank(\phi'(x,x),x) < (n+1,k)$ for any $k$ and 
		\[T_2\models  F\circ H(x)\square G(x)\iiff\phi'(x,x).\]
	\end{enumerate}
\end{proof}

\begin{lemma}\label{rank 0 k}
	Let $\vphi$ be an $x$-corrected atomic formula of rank $(0,k+1)$ for some $k\in \bN$. Then there is some quantifier free formula $\phi$ such that $\rank(\phi,x)<\rank(\vphi,x)$ and $T_2\models \vphi\iiff\phi$.
\end{lemma}
\begin{proof}
	By \Cref{f is much more}, we may assume $\vphi$ is either of the form $f^{k+1}\circ \psi(x)\square \psi(x)$ or of the form $g^{k+1}\circ \psi(x)\square \psi(x)$ for some $\psi\in \Sigma\Pi, \square\in \Set{<,>,=}$.
	\begin{enumerate}
		\item In case $\vphi$ is $f^{k+1}\circ \psi(x)\square \psi(x)$,  by \Cref{onetype}, 
		\begin{align}\label{asin1}
		T_2\models  & f^k\circ \psi(x)\square \psi(x) \iiff \\ & \label{eqA}
		\begin{bmatrix}
		\left(\psi(x)\in Z\land \bigvee_{i=0}^{k}(c_1\leq f^{i}\circ \psi(x)\leq c_2)\right) & \lor \\
		\left(\psi(x)\notin Z\land  \psi(x)=\psi(x) \right)		 
		\end{bmatrix}
		\end{align}
		and the formula in (\ref{eqA}) is of rank $(0,0)$.
		
		\item In case $\vphi$ is $g^{k+1}\circ \psi(x)\square \psi(x)$, by \Cref{fgSPAbelian}, there are finitely many constant terms $\tau_1,\dots, \tau_m$ such that 
		$f^{k+1}\circ g^{k+1}(x)=x$ for all $x\notin \Set{\tau_1,\dots,\tau_m}$.
		So 	
		\begin{align*}
		T_2\models & g^{k+1}\circ \psi(x)\square \psi(x) \iiff \\ &
		\begin{bmatrix}
		\left(\psi(x)\notin \Set{\tau_1,\dots, \tau_m}\land g^{k+1}\circ\psi(x)\square f^{k+1}\circ g^{k+1}\circ \psi(x)\right) & \lor \\
		\left(\bigvee_{i=1}^m \left(\psi(x)=\tau_i\land  g^{k+1} (\tau_i)=\tau_i \right)\right)		 
		\end{bmatrix}
		\end{align*} and $\rank\left(\bigvee_{i=0}^k (c_1\leq f^{i}\circ \psi(x)\leq c_2),\ x\right) = (-\infty, k-1)$
		
		Now, replacing $\psi(x)$ with $g^{k+1}\circ \psi(x)$ in (\ref{asin1}), we get
		\begin{align}
		T_2\models &  f^{k+1}\circ g^{k+1}\circ \psi(x)\square g^{k+1}\circ  \psi(x) \iiff \\ & \label{eqB}
		\begin{bmatrix}
		\left(g^k\circ \psi(x)\in Z\land \bigvee_{i=0}^{k}(c_1\leq f^{i}\circ g^{k+1}\circ \psi(x)\leq c_2)\right) & \lor \\
		\left(g^k \circ \psi(x)\notin Z\land  g^{k+1}\circ\psi(x)=g^{k+1}\circ \psi(x) \right)		 
		\end{bmatrix}_.
		\end{align} By noticing that $\vdash g^{k+1}\circ \psi(x)=g^{k+1}\circ \psi(x) \iiff x=x$,  the formula in (\ref{eqB}) is of rank $(0,0)$.
	\end{enumerate}
\end{proof}
\begin{lemma}\label{rank 0 0}
	Let $\vphi$ be an $x$-corrected atomic formula of rank $(0,0)$. Then there is some quantifier free formula $\phi$ such that $\rank(\phi,x)<\rank(\vphi,x)$ and $T_2\models \vphi\iiff\phi$.
\end{lemma}
\begin{proof}
	By \Cref{conseqDLOZ,DLOZinZ} we may assume $\vphi$ is of the form $\psi(x)\square x$ where $\psi\in \Sigma\Pi$ and $\square\in \Set{<,>,=}$. Now \[\vdash \left[\psi(x)\square x\right] \iiff \left[\left(\psi(x)\square x \land x\in Z\right) \lor \left(\psi(x)\square x \land x\notin Z\right) \right].\]
	By \Cref{DLOZinZ}, the right hand side is equivalent to a quantifier free formula of rank $(-\infty,-\infty)$.
\end{proof}

\begin{lemma}\label{reduceToSP}
	Let $\vphi$ be  a quantifier free formula with free variable $x$. Then there is some $x$-corrected formula $\phi$ such that $\rank(\phi, x)\leq  (-\infty,0)$ for some $k\in \bN$ and $T_2\models \vphi\iiff \phi$.
\end{lemma}
\begin{proof}
	By \Cref{reduceToFSP}  we may assume $\vphi$ is $x$-corrected. Since the lexicographic order on well-ordered sets is well-ordered,  by induction it suffices to show that if $\rank(\vphi,x)>(-\infty,0)$, then there is some $x$-corrected $\phi$ such that $\rank(\phi,x)<\rank(\vphi,x)$ and $T_2\models \vphi\iiff \phi$.
	As a Boolean combination of formulas of rank at most $(-\infty,0)$ is of rank at most $(-\infty,0)$ as well, we may further assume that $\vphi$ is atomic.
	\begin{itemize}
		\item If $\rank(\vphi,x)=(n+1,k)$ for some $n, k\in \bN$, then by \Cref{rank n+1 k} there is some quantifier free formula $\phi'$ such that $\rank(\phi',x)<\rank(\vphi,x)$.
		\item  If $\rank(\vphi,x)=(0,k+1)$ for some $k\in \bN$, then by \Cref{rank 0 k} there is some quantifier free formula $\phi'$ such that $\rank(\phi',x)<\rank(\vphi,x)$.
		\item  If $\rank(\vphi,x)=(0,0)$ for some $k\in \bN$, then by \Cref{rank 0 0} there is some quantifier free formula $\phi'$ such that $\rank(\phi',x)<\rank(\vphi,x)$.
		\item  If $\rank(\vphi,x)=(-\infty,n+1)$ for some $k\in \bN$, then by \Cref{rank infty n+1} there is some quantifier free formula $\phi'$ such that $\rank(\phi',x)<\rank(\vphi,x)$.
	\end{itemize}
	So in conclusion, whenever $\vphi$ is $x$-corrected and $\rank(\vphi,x)>(-\infty,0)$, there is some quantifier free formula $\phi'$ such that $T_2\models \vphi\iiff\phi'$ and $\rank(\phi',x)<\rank(\vphi,x)$. By \Cref{reduceToFSP}, there is some $x$-corrected formula $\phi$ such that  $T_2\models \vphi\iiff\phi'\iiff \phi$ and $\rank(\phi,x)=\rank(\phi',x)<\rank(\vphi,x)$.
\end{proof}

\begin{theorem}\label{T1QE}
	$T_2$  admits quantifier elimination.
\end{theorem}
\begin{proof} Let $\vphi(x,y_1,\dots, y_k)$ be a quantifier free $\cL_1$-formula.
	It suffices to find a quantifier free formula $\phi(y_1,\dots, y_k)$ such that $ T_2\models \exists x\vphi(x,y_1,\dots, y_k)\iiff \phi(y_1,\dots, y_k)$. By \Cref{reduceToSP}, we may assume $\vphi$ is $x$-corrected and $\rank(\vphi,x)=(-\infty, 0)$. Since $\rank(\vphi,x)=(-\infty, 0)$, there is some quantifier-free $\cL_0$-formula $\vphi'(x,z_1,\dots, z_l)$ and $\cL_1$-terms $t_1,\dots,t_l$ with variables in $\Set{y_1,\dots, y_k}$ such that \[\vphi(x,y_1,\dots, y_k) = \vphi'(x,t_1,\dots,t_k).\] Now by \Cref{T0QE}, there is some quantifier-free formula $\phi(z_1,\dots,z_k)$ such that \[T_0\models \exists x \vphi'(x,z_1,\dots,z_l)\iiff \phi(x,z_1,\dots,z_l).\]
	As $T_2\supset T_0$, in conclusion,
	\[ T_0\models  \exists x \vphi(x,y_1,\dots,y_k) \iiff \exists x \vphi'(x,t_1,\dots,t_l)\iiff \phi(t_1,\dots,t_l) \]
	and $\phi(t_1,\dots,t_l)$ is a quantifier-free $\cL_1$-formula with variables from $\Set{y_1,\dots, y_k}$.
\end{proof}

\begin{corollary}\label{T1toT2}
	Every one-variable set definable in $\cM_2$ is definable in $\cM_0$.
\end{corollary}
\begin{proof}
	By \Cref{T1QE}, every definable set in $\cM_2$ is quantifier-free definable. By \Cref{reduceToSP}, every quantifier-free one-variable set definable in $\cM_2$ is equivalent to an $x$-corrected formula of rank $\leq (-\infty,0)$, which in turn is definable (with parameters) in $\cM_0$.
\end{proof}

We conclude by articulating the answers to Questions \ref{DCTCpigeonhole} and \ref{IsThereAnAxiomatization}.
\begin{theorem}
	There is a definably complete type complete structure without the pigeonhole property.
\end{theorem}
\begin{proof}
	The failure of the pigeonhole principle in $\cM_2$ is witnessed by $Z$ and $f \upharpoonright Z$. But by \Cref{T1toT2}, $\cM_0$ and $\cM_2$ have the same definable sets in one free variable. In particular, $\cM_2$ is definably complete and type complete.
\end{proof}
\begin{theorem}
	There are two ordered structures in the same language $\cM,\cN$ on the same universe, admitting the same order and the same definable subsets with $\cM$ being pseudo-o-minimal and $\cN$ not.
	
	In particular, the answer to \Cref{IsThereAnAxiomatization} is negative and there is no axiomatization of pseudo-o-minimality by first-order conditions on one-variable formulae only. Furthermore, there is no axiomatization of pseudo-o-minimality by any second order theory in the language $\cL_{\Def}:=\Set{<,\Def}$ where $\Def$ is interpreted as the definable one-variable subsets.
\end{theorem}
\begin{proof}
	$\cM_0$ is pseudo-o-minimal and $\cM_2$ is not pseudo-o-minimal as the failure of the pigeonhole principle is witnessed by $Z$ and $f \upharpoonright Z$. But $\cM_0\upharpoonright\Set{<} = \cM_2\upharpoonright\Set{<}$ and by \Cref{T1toT2}, $\cM_0$ and $\cM_2$ have the same definable sets in one free variable. We may now define $\cN$ to be $\cM_2$ and $\cM$ to be a trivial expansion of $\cM_0$ to $\cL_1$ (letting every function symbol be interpreted as the identity map and any relation symbol be interpreted as the $\emptyset$).
	
\end{proof}

\bibliographystyle{alpha}
\bibliography{pseudo}

\end{document}